\documentclass[11pt, final]{amsart}
\usepackage{amsmath, amsxtra}
\usepackage{verbatim}

\usepackage{amssymb, mathdots}

\usepackage[notcite, notref]{showkeys}
\usepackage[colorlinks,
linkcolor=blue,
citecolor=black,
pagebackref=true]{hyperref}

\numberwithin{equation}{section}

\theoremstyle{plain}
\newtheorem{theorem}{Theorem}[section]
\newtheorem{prop}[theorem]{Proposition}
\newtheorem{corollary}[theorem]{Corollary}
\newtheorem{lemma}[theorem]{Lemma}

\theoremstyle{definition}
\newtheorem{definition}[theorem]{Definition}
\newtheorem{example}[theorem]{Example}

\def\ra{\rightarrow}

\newcommand{\gitq}{/\hspace{-0.25pc}/}
\renewcommand{\ss}{\mathrm{ss}}
\newcommand{\PP}{\mathbb{P}\,}

\newcommand{\OO}{\mathcal{O}\,}

\def\dra{\dashrightarrow}
\def\co{\colon\thinspace} 

\DeclareMathOperator{\Pic}{Pic}

\DeclareMathOperator{\spec}{Spec}
\DeclareMathOperator{\proj}{Proj}

\DeclareMathOperator{\Supp}{Supp}

\DeclareMathOperator{\Sym}{Sym}
\DeclareMathOperator{\Hilb}{Hilb}
\DeclareMathOperator{\Trig}{Trig}

\def\Hn1{\mathcal{H}_{n,1}}

\def\C{\mathcal{C}}
\def\D{\mathcal{D}}

\def\O{\mathcal{O}}

\def\M{\overline{M}}
\newcommand\Mg[1]{\overline{\mathcal{M}}_{#1}}

\def\I{\mathcal{I}}

\def\U{\mathcal{U}}

\def\EE{\mathbb{E}}

\def\PP{\mathbb{P}}

\def\CC{\mathbb{C}}
\def\GG{\mathbb{G}}

\def\HH{\mathrm{H}}

\DeclareMathOperator\SL{SL}

\def\nb{\nobreakdash}

\DeclareMathOperator{\Grass}{Grass}
\def\Cone{\text{Cone}}

\begin{document}
\title{Stability of $2^{nd}$ Hilbert points of canonical curves}
\author{Maksym Fedorchuk}
\address{Department of Mathematics, Columbia University, 2990 Broadway, New York, NY 10027}
\email{mfedorch@math.columbia.edu}
\author{David Jensen}
\address{Department of Mathematics, SUNY Stony Brook, Stony Brook, NY 11794}
\email{djensen@math.sunysb.edu}
\begin{abstract} We establish GIT semistability of the $2^{nd}$ Hilbert point of every
Gieseker-Petri general canonical curve by a simple geometric argument.
As a consequence, we obtain an upper bound on slopes
of general families of Gorenstein curves. We also explore the question of what replaces hyperelliptic curves in the
GIT quotients of the Hilbert scheme of canonical curves.
\end{abstract}
\maketitle

\setcounter{tocdepth}{1}
\tableofcontents

\section{Introduction}
\label{S:intro}
The log minimal model program for $\M_g$, also known as the {\em Hassett-Keel program}, offers a promising approach to understanding the birational geometry of the moduli space of curves.
The goal of this program is to find a functorial interpretation of the log canonical models
\[
\M_g(\alpha)=\proj \bigoplus_{m\geq 0} \HH^0(\Mg{g}, \lfloor m(K_{\Mg{g}}+\alpha\delta)\rfloor).
\]
Such an interpretation could then be used to study properties of the rational contraction $\M_g \dra \M_g(\alpha)$ and 
to obtain structural results about the cone of effective divisors of $\M_g$, in particular its
Mori chamber decomposition.

Hassett and Hyeon constructed the first two log canonical models of $\M_g$ by considering 
GIT quotients of asymptotically linearized Hilbert schemes of tricanonical and 
bicanonical curves \cite{hassett-hyeon_contraction,hassett-hyeon_flip}. 
It is widely expected that further progress in the Hassett-Keel 
program will require GIT stability analysis
of finite (i.e. non-asymptotic) Hilbert points of bicanonical and canonical curves;
see \cite{morrison-git,morrison-swinarski,handbook,alper-hyeon}. The case of canonical curves is
of particular interest because it should lead to birational contractions of $\M_g$ affecting the interior.

Only recently
it was shown that finite Hilbert points of general canonical curves are semistable in all genera \cite{AFS-odd,AFS-even}.
In particular, the question of {\em which} smooth canonical curves have (semi)stable $m^{th}$ Hilbert points
is widely open. Here we make partial progress towards answering this question.
Our main result is a geometric proof of the semistability of the $2^{nd}$ Hilbert point of a general canonical curve
that leads to a sufficient condition for semistability, something that the previous results lack.
\begin{theorem}\label{T:git-semistability} Let $C$ be a Gieseker-Petri general smooth curve of genus $g\geq 4$.
Then its canonical embedding $C \subset \PP^{g-1}$ has semistable $2^{nd}$ Hilbert point.
\end{theorem}
This result strengthens and complements the proof of generic semistability of finite Hilbert points of canonical
curves in \cite{AFS-odd,AFS-even}. Not only do we show that the GIT quotient of the variety of $2^{nd}$ Hilbert points
of canonical curves is non-empty, but also that this GIT quotient
parameterizes all curves whose linear systems behave generically. 
Assuming the expected {\em stability} of the general canonical curve, this GIT quotient is
an interesting projective birational model of $\M_g$:

If $G$ is the quotient in question, then the map $f\co \M_g \dra G$ is not a local isomorphism along the locus of
curves of low Clifford index. 
As we show in this paper, $f$ is not regular along the hyperelliptic locus $\overline{H}_g\subset \M_g$ 
(see Section \ref{S:riddle}). In addition, $f$ is not regular along the locus $\Trig_g(+)$ of trigonal curves with positive Maroni invariant
and contracts the locus $\Trig_g(0)$ of trigonal curves
with Maroni invariant $0$ to a point (this locus is non-empty only for even $g$); see Corollary \ref{C:trigonal}.
We also observe that $f$ is not regular along the bielliptic locus for $g \geq 7$. Finally, in the case $g=6$, the rational map
$f$ contracts both the bielliptic locus (see Proposition \ref{P:bielliptic}) and the locus of plane quintics 
(see Corollary \ref{C:quintics}).

In addition to studying the indeterminacy locus of the map $f\co \M_g \dra G$, we also examine the
indeterminacy locus of its inverse $f^{-1}\co G \dra \M_g$. To this end, we show that
$G$ parameterizes many different types of singular curves, a large number of which are enumerated
in Theorem \ref{T:singularities}.  Each of these singularities is predicted to play a role in a functorial interpretation
of $\M_g(\alpha)$; see \cite{AFS-modularity} for precise predictions.
As a consequence of our analysis, we discover a class of curves, the {\em $A_{2g}$-rational curves},
which lie in the total transform under $f$ of
the hyperelliptic locus $\overline{H}_g$.

Finally, we include an important application of our semistability result, 
providing an upper bound on slopes of one-parameter
families of Gorenstein curves with a sufficiently general generic fiber.
\begin{theorem}\label{T:CH-inequality}
Let $B$ be a complete curve. If $\C \ra B$ is a flat family of Gorenstein curves with generic fiber a canonically embedded curve
whose $2^{nd}$ Hilbert point is semistable,
then the degree $\lambda$ of the Hodge bundle and the degree $\delta$ of the discriminant divisor satisfy the inequality
\begin{align*}
\frac{\delta}{\lambda}\leq 7+\frac{6}{g}.
\end{align*}
\end{theorem}
This theorem is an extension of a celebrated result of Cornalba and Harris \cite{CH} saying
that the slope of any generically smooth family of Deligne-Mumford stable curves of genus $g$ is at most $8+4/g$.
We prove Theorem \ref{T:CH-inequality}
in Section \ref{S:slope-inequality}, where we explicate the genericity assumption and give a precise
definition of $\lambda$ and $\delta$. 

\section{Semistability of $2^{nd}$ Hilbert points}
\label{S:$2^{nd}$-hilbert}

We briefly recall the necessary definitions. Let $C\hookrightarrow \PP^{g-1}$ be a canonically embedded smooth curve of genus $g\geq 4$.
Using Max Noether's theorem on projective normality of canonical curves \cite[p.117]{ACGH},
we define the $2^{nd}$ Hilbert point of 
$C\hookrightarrow \PP^{g-1}$ to be the quotient
\[
\left[\HH^0(\PP^{g-1}, \O_{\PP^{g-1}}(2)) \ra \HH^0(C,\O_C(2)) \ra 0\right]\in
 \Grass\left(3g-3, \binom{g+1}{2}\right).
 \]
We denote by $\overline{\Hilb}^{\, 2}_{g}$ the closure of the
locus of $2^{nd}$ Hilbert points of canonically embedded curves in the Grassmannian $\Grass(3g-3, \binom{g+1}{2})$ and endow $\overline{\Hilb}^{\, 2}_{g}$
with the linearization $\OO(1)$ coming from the Pl\"{u}cker embedding of the Grassmannian into
$\PP \bigwedge^{3g-3} \HH^0(\PP^{g-1}, \O_{\PP^{g-1}}(2)).$ Finally, we set
\[
G:=\overline{\Hilb}^{\, 2, \ss}_{g} \gitq \SL(g) = \proj \bigoplus_{m\geq 0} \HH^0(\overline{\Hilb}^{\, 2}_{g}, \O(m))^{\SL(g)}
\]
to be the resulting GIT quotient.
One reason that this construction is of particular interest is that the map
\begin{equation*}
f\co \M_g \dashrightarrow G
\end{equation*}
is not an isomorphism on the interior $M_g \subset \M_g$.
More precisely, we show that curves of Clifford index $0$ and $1$ are outside of the locus where this map is locally an isomorphism.
We consider hyperelliptic, trigonal, and bielliptic curves in the later sections of the paper.

We proceed to state the main result of our
paper in its greatest generality and to record its most important corollaries.
\begin{theorem}\label{T:semistability}
A canonical curve of genus $g$ not lying on a quadric of rank $3$ or less has semistable $2^{nd}$ Hilbert point.
\end{theorem}

\begin{proof} 
Our key geometric tool is the $\SL(g)$\nb-invariant effective divisor $\D \subset \Grass(3g-3, \binom{g+1}{2})$ defined as
the locus of $(3g-3)$-dimensional quotients
of $\HH^0(\PP^{g-1}, \O_{\PP^{g-1}}(2))$ whose kernel contains a quadric of rank  at most $3$.
The fact that $\D$ is a divisor follows directly from the fact that the locus of quadrics of rank at most $3$
has dimension $3g-3$ in $\HH^0(\PP^{g-1}, \O_{\PP^{g-1}}(2))$.
Since $\Grass(3g-3, \binom{g+1}{2})$ is smooth of Picard number $1$, the divisor $\D$ is defined by a global section of some power of $\O(1)$.  Since $\SL(g)$ has no non-trivial characters, this section is $\SL(g)$-invariant.
It follows that any curve whose Hilbert point is not contained in $\D$ is semistable.
\end{proof}

Recall that a complete smooth curve $C$ is said to be {\em Gieseker-Petri general}
if it satisfies the Petri condition that
\[
\mu\co \HH^0(C,L)\otimes \HH^0(C,K_C-L) \ra \HH^0(C,K_C)
\]
is injective for all $L\in \Pic(C)$. That a general curve in $M_g$ is Gieseker-Petri
general was proved by Gieseker \cite{Gieseker-Petri}, as well as Eisenbud and
Harris \cite{EHPetri} using degeneration arguments.

\begin{lemma}
\label{L:rank-3-quadric}
The canonical embedding of a
Gieseker-Petri general curve does not lie on a rank $3$ quadric.
\end{lemma}
\begin{proof}
Suppose a canonically embedded curve lies on a quadric
of rank $3$ whose vertex is a linear space $\Lambda$ of dimension
$g-3$. The projection away from $\Lambda$ maps $C$ onto a conic
$R\simeq \PP^1$ in
$\PP^2$. It follows that there is a decomposition $K_C=2L+B$.
Here, $B$ is an effective divisor with
$\Supp(B)=\Lambda\cap C$, and $L$ is a pullback of $\O(1)$ from $R$.
In particular, we have $h^0(C,L)\geq 2$. Let $s_0, s_1$ be two distinct
non-zero global sections of $L$, and $s'_0, s'_1$ be the same rational functions
considered now as sections of $L+B$. Then
$$\mu(s_0\otimes s_1'-s_1\otimes s_0')=0,$$
violating the Petri condition.
\end{proof}
\begin{corollary}
Theorem \ref{T:git-semistability} holds.
\end{corollary}
\begin{proof}
Indeed, Theorem \ref{T:git-semistability} is an immediate corollary of Lemma \ref{L:rank-3-quadric} and Theorem \ref{T:semistability}.
\end{proof}

\section{Degenerations to rational normal surface scrolls}
\label{S:Degenerations}

Aside from canonical curves, there is another variety of interest in $\PP^{g-1}$ with ideal generated by
$\binom{g-2}{2}$ quadrics, namely a rational normal surface scroll.
Recall that for non-negative integers $a$ and $b$ satisfying $a+b=g-2$,
a rational normal surface scroll $S_{a,b} \subset \PP^{g-1}$ is the join
of two rational normal curves of degrees $a$ and $b$ whose linear spans do not intersect.

Rational normal surface scrolls are of particular interest to us because
the linear system of quadrics containing a smooth trigonal canonical curve $C \subset \PP^{g-1}$ cuts out
precisely such a surface. Namely, by the geometric Riemann-Roch the $g^1_3$'s on $C$ are collinear in $\PP^{g-1}$,
and the resulting lines sweep out a rational normal surface scroll $S_{a,b}$.
The difference $|a-b|$ is classically known as the {\em Maroni invariant} of $C$.
An important fact is that the ideal of the rational normal surface containing $C$
is generated by the quadrics containing
$C$ \cite{ACGH}.
It follows that the $2^{nd}$ Hilbert point of a smooth trigonal canonical curve $C \subset \PP^{g-1}$ coincides with the
$2^{nd}$ Hilbert point of the rational normal scroll containing it.

In this section, we show that a rational normal surface scroll $S_{a,b}$ is semistable if and only if
$a=b$, i.e. if it is a $\PP^1\times\PP^1$
embedded by the complete linear system $\vert \O_{\PP^1\times\PP^1}(1,a)\vert$ in $\PP^{2a+1}$.
\begin{prop}\label{P:scroll-stability}
A rational normal surface $S_{a,b}$ has semistable $2^{nd}$ Hilbert point if and only if $a=b$.
\end{prop}

\begin{proof}
The scroll $S_{a,a}$ in $\PP^{2a-1}$ is the image of the homogeneous space $\PP^1 \times \PP^1$ embedded
via the complete linear system $\vert \OO (a-1,1) \vert$.
The fact that its $m^{th}$ Hilbert point is semistable now follows from Kempf's stability results \cite[Corollary 5.3]{Kempf}.

Suppose now that $a\neq b$.
To see that the scroll $S_{a,b}$ is non-semistable, recall that the ideal of $S_{a,b}$ is
generated by the determinants of the $2 \times 2$ minors
of the following matrix
\[
 \left(\begin{array}{cccc|cccc}
x_{0} & x_{1} & \cdots & x_{a-1} & y_{0} & y_{1} & \cdots & y_{b-1} \\
 x_{1} & x_{2} & \cdots & x_{a} & y_{1} & y_{2} & \cdots & y_{b}
 \end{array}\right)
 \]

We consider the one-parameter subgroup of $\operatorname{Aut}(S_{a,b})\subset \SL(g)$
acting with weight $-(b+1)$ on $x_i$'s and weight $a+1$ on $y_i$'s.
The ideal of $S_{a,b}$ becomes homogeneous with respect to $\rho$. It follows that every monomial
basis of $\HH^0(S_{a,b}, \O(m))$ has the same $\rho$-weight. For $m=2$, we compute that the $\rho$-weight of
$\HH^0(S_{a,b}, \O(2))$ is
\[
2(b+1)\binom{a}{2}-2(a+1)\binom{b}{2}-ab(a-b)=(a-b)(a+b-1) \neq 0.
\]
Since the $\rho$-weight is non-zero, we conclude that $S_{a,b}$ is non-semistable.
\end{proof}

As a corollary, we obtain the following two results.

\begin{corollary}\label{C:trigonal}
Trigonal curves with positive Maroni invariant have non-semistable $2^{nd}$
Hilbert points. Trigonal curves with Maroni invariant $0$ are strictly semistable and are identified in 
$G=\overline{\Hilb}^{\, 2, \ss}_{g} \gitq \SL(g)$ with the point corresponding to the balanced rational normal
surface scroll in $\PP^{g-1}$.
\end{corollary}
\begin{proof}
This follows immediately from Proposition \ref{P:scroll-stability} and the fact that the $2^{nd}$ Hilbert
point of a trigonal curve of genus $g$ and Maroni invariant $m$ coincides with the $2^{nd}$ Hilbert
point of the scroll $S_{\frac{g+m}{2}-1, \frac{g-m}{2}-1}$.
\end{proof}
We note that non-semistability of trigonal curves with a positive Maroni invariant
reflects the fact that the locus in $\overline{M}_g$ of trigonal curves contained in an unbalanced scroll is covered by families of slope strictly greater than $7 + \frac{6}{g}$; in particular, when $g$ is odd, $\Trig_g$
is covered by families of slope $7+\frac{20}{3g+1}$ \cite{anands}.

The second corollary of Proposition \ref{P:scroll-stability} shows that 
$G=\overline{\Hilb}^{\, 2, \ss}_{g} \gitq \SL(g)$
parameterizes curves with numerous singularities, as predicted by \cite{AFS-modularity}.
\begin{theorem} \label{T:singularities}
Suppose $g$ is even. There exist non-trigonal canonical curves with semistable $2^{nd}$ Hilbert point and possessing
the following classes of singularities:
\begin{enumerate}
\item[(a)] All $A_k$ singularities with $k\leq 2g+1$.
\item[(b)] All $D_k$ singularities with $k\leq 2g$.
\item[(c)] If $g=6m-2$, the singularity $y^3=x^{g+2}$ and its deformations.
\item[(d)] If $g=6m-4$, the singularity $y^3=x^{g+1}$ and its deformations.
\item[(e)] If $g=6m$, the singularity $y^3=x^{g+1}$ and its deformations.
\end{enumerate}
\end{theorem}
\begin{proof} Let $g=2k$. We begin by constructing a curve $C$, with a desired singularity $p\in C$,
in the class $(3,k+1)$ on $\PP^1\times \PP^1$.
Next we embed $C$ via the restriction of the complete linear system $\vert \O_{\PP^1\times\PP^1}(1,k-1)\vert$, which is evidently a canonical
linear system on $C$.
The $2^{nd}$ Hilbert point of the canonical embedding of $C$ will then be the $2^{nd}$ Hilbert point
of the balanced normal scroll $S_{k-1,k-1}$, hence semistable by Proposition \ref{P:scroll-stability}.
We can then deform $C$ out of the scroll preserving singularities of $C$
and the semistability of its $2^{nd}$ Hilbert point.
\subsubsection*{Construction of the singular curve $(C,p)$ on the scroll:} \hfill
\par
\smallskip
\noindent
(a)
Consider a smooth rational curve $C_1$ in the class $(2,1)$ on $\PP^1\times \PP^1$.
Since $h^0(\OO_{C_1}(-1,k-1))=h^1(\OO_{C_1}(-1,k-1))=0$, the restriction
map $\vert \O_{\PP^1\times \PP^1}(1,k)\vert \ra \vert \O_{C_1}(2k+1)\vert$ is bijective.
It follows that for every $p\in C_1$, there exists a unique divisor $C_2\in \vert \OO_{\PP^1\times \PP^1}(1,k)\vert$
such that $(C_1\cdot C_2)_p=(2k+1)$. Evidently, such a divisor is smooth if $p$ is not a ramification
point of the projection $\operatorname{pr}_2\co C_1 \ra \PP^1$.

It follows that for the general point $p\in C_1$, there is a smooth rational curve
$C_2\in \vert \O_{\PP^1\times \PP^1}(1,k)\vert$ such that $C_1$ and $C_2$ are maximally tangent at $p$.
Namely, we have $(C_1\cdot C_2)_p=(2k+1)$. It follows that $C:=C_1\cup C_2$ is a curve
of class $(3, k+1)$
on $\PP^1\times\PP^1$
with a unique singularity of
type $A_{2g+1}$. The complete linear system
$\vert \OO (1,k-1) \vert$ embeds $\PP^1 \times \PP^1$ in $\PP^{g-1}$, mapping $C_1$ and $C_2$ to rational normal curves, meeting in a singularity of type $A_{2g+1}$ at $p$. Thus the image of $C$ under this embedding
is an $A_{2g+1}$-rational curve on $S_{k-1,k-1}$; see Definition \ref{D:A-curves}. \\
(b) A curve with a $D_{2g}$ singularity is obtained by taking a nodal curve $C_1$ of class $(2,2)$ and
a curve $C_2$ of class $(1,k-1)$ that is tangent with multiplicity $2k-1$ to one of the branches at the node of $C_1$. \\
(c) If $k=3m-1$, then we take three rational curves in the class $(1,m)$, all meeting at a single point where they pairwise intersect with multiplicity $2m$. The resulting singularity is analytically isomorphic to $y^3=x^{g+2}$. \\
(d) may be proved analogously to Part (c). \\
(e) We exhibit an explicit curve in the class $(3,3m+1)$ with singularity analytically isomorphic to $y^3=x^{6m+1}$.
Namely, consider
\begin{equation}
(y-x^m)^3-x^{3m+1}y^3=0.
\end{equation}
This curve has a rational parameterization $x=t^3$, $y=t^{3m}/(1-t^{3m+1})$. Evidently, under this parameterization $x=t^3$ and $y-x^m=t^{6m+1}+\cdots$. The claim follows.

Having established the existence of a curve $C$, with a desired singularity $p\in C$, 
in the class $(3,k+1)$ on $\PP^1\times \PP^1$, we must now show that there exists a non-trigonal
semistable curve with the same singularity.  To do this, we observe that a general equisingular deformation of $C$ in $\PP^{g-1}$ is non-trigonal. More precisely, the
deformations of $C$ as a subscheme of $\PP^{g-1}$ and the deformations of $C$ as a $(3,k+1)$ divisor on 
$\PP^1\times \PP^1$ both surject smoothly onto the
deformation space of the singularity $p\in C$. Since the dimension of the Hilbert scheme of canonical curves is
$(3g-3)+(g^2-2g)$, the dimension of the $\SL(g)$-orbit of the scroll is $g^2-2g-6$, and the dimension 
of the linear system $\vert \OO_{\PP^1\times \PP^1}(3,k+1)\vert$ is 
$2g+7$, we conclude that the general equisingular deformation of $(C,p)$ does not lie on the scroll if and only if
$3g-3> 2g+1$, or $g> 4$. This concludes the proof.
\end{proof}

In the specific case of $g=6$, there is another surface of interest -- the Veronese surface:  If $C \subset \PP^5$ is a
smooth canonical curve of genus $6$ that admits a $g^2_5$, then any $5$ points in a $g^2_5$ are coplanar by
the geometric Riemann-Roch. It follows that each of the quadrics containing $C$ also contains the conic 
through these five points. The resulting
two-dimensional family of conics sweeps out the Veronese surface in $\PP^5$. Moreover, the ideal of the Veronese
surface is generated by the quadrics containing $C$.

\begin{prop}
The Veronese surface in $\PP^5$ has semistable $2^{nd}$ Hilbert point.
\end{prop}

\begin{proof}
This also follows immediately from \cite[Corollary 5.3]{Kempf},
as the Veronese surface is simply $\PP^2$ embedded in $\PP^5$ via the complete linear system $\vert \OO_{\PP^2}(2) \vert$.
\end{proof}
\begin{corollary}\label{C:quintics}
A canonically embedded plane quintic has semistable $2^{nd}$ Hilbert point, coinciding with the $2^{nd}$ Hilbert point of a Veronese
surface.
\end{corollary}

\section{An answer to the riddle}
\label{S:riddle}
What is the limit of the canonical model of a smooth curve as it degenerates to a hyperelliptic curve?  This is the question that opens a well-known paper of Bayer and Eisenbud \cite{BE}.  In this section, we aim to show that their answer -- a ribbon -- is only part of the story.  In fact, there is a larger class of curves, the $A_{2g}$-rational curves, that give a canonical answer to this question, at least from the point of view of GIT for canonical curves.

\begin{definition}\label{D:A-curves}
A complete connected reduced curve of genus $g$ with a unique singularity of type $A_{2g} \ (y^2=x^{2g+1})$ is called an {\em $A_{2g}$-rational curve}. A complete connected reduced curve of genus $g$ with a unique singularity of type $A_{2g+1} \ (y^2=x^{2g+2})$ is called an {\em $A_{2g+1}$-rational curve}.
\end{definition}
Since the genus of the two singularities $A_{2g+1}$ and $A_{2g}$ both equal $g$, an $A_{2g}$-rational curve is
necessarily irreducible and its normalization is isomorphic to $\PP^1$. Similarly, an
$A_{2g+1}$-rational curve necessarily has two irreducible components, each isomorphic to $\PP^1$.
We will denote an $A_{2g+1}$-rational curve $C$ with the singularity $\hat{\O}_{C,p}\simeq k[[x,y]]/(y^2-x^{2g+2})$ by $(C,p)$.

Isomorphism classes of $A_{2g+1}$-rational curves
with a fixed pointed normalization are
in bijection with closed points of $\GG_m \times \GG_a^{g-1}$.
Indeed, let the pointed normalization of an $A_{2g+1}$-rational curve be a disjoint union of two pointed rational curves $(\PP^1, p_1)$ and $(\PP^1, p_2)$,
where the uniformizer at $p_1$ is $x$ and at $p_2$ is $y$. Then the isomorphism class of a (parameterized) $A_{2g+1}$-curve is specified by a gluing datum
$y\mapsto a_1x+\cdots+a_gx^{g}$, where $a_1\neq 0$,
that defines an isomorphism
\[
\CC[y]/(y^{g+1}) \ra \CC[x]/(x^{g+1})
\]
along which the two length $g+1$ subschemes supported at $p_1$ and $p_2$ respectively are glued.
We call $(a_1,a_2,\dots,a_g) \in \GG_m \times \GG_a^{g-1}$ the {\em crimping}, and refer the reader to \cite{fred}
for a systematic treatment of crimping for singular curves.

Suppose that $C$ is an
$A_{2g+1}$-rational curve given by the gluing datum $y\mapsto a_1x+\cdots+a_gx^{g}$.
Since $C$ is a local complete intersection curve, it admits a dualizing line bundle
$\omega_C$. While there are numerous ways to get a handle on this line bundle, we will only consider
the one that, to us, is the most explicit. 
Namely, we use the defining property which says that $\omega_C$ is the unique line bundle that
restricts to $\O(g-1)$ on each irreducible rational component and has $g$ global sections.
It follows that we can identify $K_C$ with the triple $(\O(g-1), \O(g-1), \kappa_C)$
where
\[
\kappa_C=1+k_1x+\cdots+k_{g}x^{g} \in \left(\CC[x]/(x^{g+1})\right)^*
\] is a gluing datum for a line bundle on $C$.
Thus the determination of $\omega_C$ reduces to computing $\kappa_C$.
\begin{prop}\label{P:canonical}
The canonical line bundle $\omega_C$ is defined by
\[
\kappa_C=1+k_1x+\cdots+k_{g-1}x^{g-1}+k_gx^{g},
\] where $k_g=0$ and $k_i$, $1\leq i\leq g-1$, are (uniquely determined) polynomials in $a_1,(a_1)^{-1},a_2,\dots, a_g$.
\end{prop}
\begin{proof}
Since $\omega_C\vert_{C_1}\simeq \O(g-1)$ has exactly $g$ global sections $1,y,\dots, y^{g-1}$, all of them
have to lift to global sections of $\omega_C$. This means that each of the elements $\kappa_C, \kappa_C y,
\dots, \kappa_C y^{g-1}$ of $\CC[x]/(x^{g+1})$ must be a linear combination of $1,x,\dots,x^{g-1}$.
From this, we immediately obtain that $k_{g}=0$. Next, setting to $0$ the coefficient of $x^{g}$ in
\[
\kappa_C y^n=(1+k_1x+\cdots+k_{g-1}x^{g-1})(a_1x+\cdots+a_{g}x^{g})^n ,
\]
we obtain
\begin{equation}\label{rel}
a_1^nk_{g-n}+na_1^{n-1}a_2k_{g-n-1}+\cdots =0.
\end{equation}
Setting $n=g-1$, this gives $k_1=-na_2/a_1$, which determines $k_1$ uniquely. The assertion for $k_{2}, \dots, k_{g-1}$
follows by induction by applying \eqref{rel} for $n=g-2, \dots, 1$ repeatedly.
\end{proof}
\begin{example}[{see \cite[Section 2.3.7]{fedorchuk-genus4}}]
Up to projectivities,
there is a unique canonically embedded $A_9$-curve $C$. It can be defined
by the crimping datum $y\mapsto x+x^2+x^3+x^4$. A quick computation shows that
the gluing datum of $\omega_C$ is $\kappa_C=1-3x+5x^2-5x^3$. It follows that $C$ can be
defined parametrically
by
\begin{multline*}
x_0=\left(\begin{matrix} 1 \\ 1-3x+5x^2-5x^3\end{matrix}\right), x_1=\left(\begin{matrix} y \\ x-2x^2+2x^3\end{matrix}\right), \\
x_2=\left(\begin{matrix} y^2 \\ x^2-x^3\end{matrix}\right), x_3=\left(\begin{matrix} y^3 \\  x^3\end{matrix}\right).
\end{multline*}
\end{example}
\begin{example}\label{E:example-odd}
Suppose $g=2k+1$. Consider the crimping datum \[y\mapsto x-tx^{k+2},\] where $t\neq 0$ is a parameter.
One easily computes that $\kappa_C=1+tkx^{k+1}$ and that the following is a basis of $\HH^0(C,\omega_C)$:
\begin{multline*}
\omega_i=(x^i+t(k-i)x^{k+1+i}, y^i), \ i=0,\dots, k-1, \\ \omega_i=(x^i,y^i), \ i=k, \dots, 2k.
\end{multline*}
\end{example}
We recall the definition of the balanced ribbon of genus $g=2k+1$ from \cite{AFS-odd}: it is a canonical ribbon
obtained by gluing $\spec \CC[u,\epsilon]/(\epsilon^2)$ and $\spec \CC[v,\eta]/(\eta^2)$ via the isomorphism
\begin{align*}
u &\mapsto v^{-1}+v^{-k-2}\eta, \\
\epsilon&\mapsto v^{-g-1}\eta
\end{align*}
of distinguished open affines $\spec \CC[u,u^{-1},\epsilon]/(\epsilon^2)$ and $\spec \CC[v,v^{-1},\eta]/(\eta^2)$.
\begin{lemma}\label{L:ribbon-deforms} The flat limit as $t\to 0$ of the $A_{2g+1}$-curve in Example \ref{E:example-odd} is
the balanced canonical ribbon $R$ of genus $g=2k+1$.
\end{lemma}
\begin{proof} 
Recall from \cite[Lemma 3.1]{AFS-odd} that there is a basis of $\HH^0(R,\omega_R)$ whose elements
can be identified with the following polynomials in $u$ and $\epsilon$ (here $\epsilon^2=0$):
\[
z_i=u^{i}, \ 0\leq i \leq k, \qquad z_{i}=u^{i}+(i-k)u^{i-k-1}\epsilon, \ k+1\leq i \leq 2k.
\]
We keep the notation of Example \ref{E:example-odd}. Note that if we set $\psi_i:=\omega_i/(x^{2k}, y^{2k})$ and
$w:=1/x$, then 
\begin{multline*}
\psi_i=(w^i, y^{-i}), \ 0\leq i\leq k, \\ \psi_{i}=(w^{i}+(i-k) w^{i-k-1}t, y^{-i}),\ k+1\leq i \leq 2k.
\end{multline*}
To prove the lemma, it suffices to show that any quadratic relation among the $z_i$'s is a flat limit of a quadratic relation
among the $\psi_i$'s as $t\to 0$.
If we evaluate a quadratic relation among $z_i$'s on $\psi_i$, we obtain an expression of the form $(f(w) t^2, 0)$,
where $\deg f(w)\leq 2k-2$. It remains to show that $(w^i t, 0)$ can be obtained as
a quadratic polynomial in the $\psi$'s for $0\leq i \leq 2k-2$. Indeed, we have
\begin{multline*}
(w^i t, 0)=(w^{k+i+1}+(i+1) w^i t, y^{-k-i-1})(1,1)-(w^{k+i}+i w^{i-1} t, y^{-k-i}) (w, y^{-1}) \\ =\psi_{k+i+1}\psi_0-\psi_{k+i}\psi_1, \
\text{for $0\leq i\leq k-1$},
\end{multline*}and
\begin{multline*}
(w^i t, 0)=(w^{2k}+k w^{k-1} t, y^{-2k})(x^{i-k+1},y^{-i+k-1})\\-(w^{2k-1}+i w^{k-2} t, y^{-2k+1}) (w^{i-k+2}, y^{-i+k-2})
 \\ =\psi_{2k}\psi_{i-k+1}-\psi_{2k-1}\psi_{i-k+2},\
\text{for $k\leq i\leq 2k-2$}.
\end{multline*}
\end{proof}

We conclude with an observation that the general $A_{2g+1}$-rational curve is semistable.
We would prefer the stronger statement that such a curve is in fact stable, but at present we have no proof.

\begin{prop} \label{P:A-2g+1-stable}
A general $A_{2g+1}$-rational curve has semistable $2^{nd}$ Hilbert point.
\end{prop}

\begin{proof}
By the above, the variety of $A_{2g+1}$-rational curves in $\PP^{g-1}$ is irreducible. 
Thus, it suffices to find a single $A_{2g+1}$-rational curve with semistable $2^{nd}$ Hilbert point. When $g$ is even,
this is already done by Theorem \ref{T:singularities} (a).
In the case of odd genus, we recall from \cite{AFS-odd} that the balanced canonical ribbon $R$ has semistable $2^{nd}$
Hilbert point. Since $R$ deforms flatly to $A_{2g+1}$-rational curves by Lemma \ref{L:ribbon-deforms}, we are done.
\end{proof}

\begin{corollary}
\label{C:A-2g-stable}
A general $A_{2g}$-rational curve is semistable.
\end{corollary}
\begin{proof}
The general $A_{2g}$-rational curve is a deformation of the general $A_{2g+1}$-rational curve. The statement now
 follows from Proposition \ref{P:A-2g+1-stable}.
\end{proof}

\section{A slope inequality apr\`{e}s Cornalba and Harris}\label{S:slope-inequality}
\def\wtw{\widetilde{\omega}}

In this section, we prove Theorem \ref{T:CH-inequality}.
To set notation, consider a flat proper family $\pi\co \C \ra B$
of Gorenstein curves of arithmetic genus $g\geq 4$.
By assumption, the relative dualizing sheaf $\omega:=\omega_{\C/B}$ is a line bundle. We set
\begin{align*}
\lambda :=c_1(\pi_*\omega), \qquad \lambda_2 :=c_1(\pi_*\omega^2).
\end{align*}
After a finite base change,
we will assume that $\lambda$ is divisible by $g$ in $\Pic(B)$ and we let $\wtw:=\omega(-\lambda/g)$.
Then the {\em normalized Hodge bundle} $\EE:=\pi_* \widetilde{\omega}$
has a trivial determinant, i.e. the transition matrices of $\EE$ are given by elements of $\SL(g, \O_B)$.

\subsubsection*{Line bundles on the moduli stack of Gorenstein curves}
Let $\U_g$ be the irreducible component of the stack of all complete
canonically polarized Gorenstein curves of arithmetic genus $g$ that parameterizes smoothable curves.
Then $\lambda$ and $\lambda_2$ are well-defined line bundles on $\U_g$.
We formally define $\delta:=13\lambda-\lambda_2$. Note that $\Mg{g}\subset \U_g$ is an open
immersion and that on $\Mg{g}$ the line bundle $\delta$ has a geometric interpretation
as the line bundle associated to the Cartier
divisor of nodal curves. Under certain conditions this geometric interpretation can be extended to a larger open of $\U_{g}$.
To do this, we consider the regular
locus $\U_g^{\text{reg}}\subset \U_g$
and define $\Delta:=\U_g^{\text{reg}}\smallsetminus \mathcal{M}_g$ to be the locus
parameterizing singular curves. Let $\Delta'$ be the union of those irreducible components of $\Delta$ whose generic
points parameterize worse than nodal curves. Then on $\U_g^\circ:=\U_g^{\text{reg}}\smallsetminus \Delta'$
the irreducible components of $\Delta$ are Cartier divisors whose generic points parameterize nodal curves.
By construction, the locus of worse than nodal curves in $\U_g^\circ$ is of codimension at least two. Thus
the relation $\O(\Delta)=13\lambda-\lambda_2$ extends from $\Mg{g}$ to $\U_g^\circ$. We conclude that at least
on $\U_{g}^{\circ}$, the formally defined line bundle $\delta$ is the associated line bundle of the Cartier divisor
$\Delta\subset \U_{g}^{\circ}$ parameterizing singular curves.


\subsubsection*{Slopes of families of Gorenstein curves}
Given a family $\C\ra B$ as above, with $B$ a complete curve, we define the {\em slope} of $\C\ra B$ to be
$(\delta\cdot B)/(\lambda\cdot B)$.
We proceed to prove Theorem \ref{T:CH-inequality}, which is
a generalization of a well-known result of Cornalba and Harris saying that the slope of an {\em arbitrary}
generically smooth family of stable curves of genus $g$ is at most $8+4/g$ \cite{CH}. Note
that the Cornalba-Harris theorem is sharp: The general family of hyperelliptic curves of genus $g$ has slope precisely $8+4/g$,
while there are families of bielliptic curves of slope $8$, as Example \ref{E:bielliptic}
shows, and there are families of trigonal curves
of slope $36(g+1)/(5g+1)$ by \cite{stankova,AFS-modularity}.

\begin{proof}[Proof of Theorem \ref{T:CH-inequality}:]
Our proof follows closely the original argument of Cornalba and Harris. We include details for completeness.

The key input in \cite{CH} is the
{\em asymptotic} Hilbert semistability of the canonically embedded general fiber of $\C\ra B$.
Our assumption that
the general fiber $C$ has a semistable $2^{nd}$ Hilbert point is much stronger than asymptotic
semistability and so leads to a stronger inequality. On the other hand, not every smooth
canonical curve has a semistable $2^{nd}$ Hilbert point (see Proposition \ref{P:bielliptic}), so
while our inequality is stronger, we have to settle for a genericity assumption on the generic fiber.

By definition, semistability of $C$ is equivalent to the existence
of an $\SL(g)$-invariant polynomial $f\in \HH^0(\PP W, \O(d))$, where $W=\bigwedge^{3g-3} \Sym^2 \HH^0(C,\omega_C)$, 
that does not vanish at the point
$$
\bigwedge^{3g-3} \Sym^2 \HH^0(C,\omega_C) \ra \bigwedge^{3g-3} \HH^0(C,\omega^{2}_C)
$$
of $\PP W$. Under the usual identification $\HH^0(\PP W, \O(d))\simeq \Sym^d W$, the polynomial $f$ corresponds to a section of $\Sym^d W$ that maps to a non-zero section
of $\Sym^d \bigwedge^{3g-3} \HH^0(C,\omega^{2}_C)$.

Consider now the family $\pi\co \C \ra B$ as in the statement of the theorem. Let $\EE=\pi_*(\wtw)$ be the normalized
Hodge bundle. Here, $\wtw=\omega(-\lambda/g)$ and so $\det \EE\simeq \O_B$.
Using the $\SL(g)$-invariance of $f$ and the fact that transition matrices of $\Sym^{d} \bigwedge^{3g-3} \Sym^2 \EE$ 
correspond to an $\SL(g)$ coordinate change,
we conclude that there exists a section of $\Sym^{d} \bigwedge^{3g-3} \Sym^2 \EE$ that maps to a generically non-vanishing section
of $\Sym^d \bigwedge^{3g-3} \pi_*\left(\widetilde{\omega}^{2}\right)$.
Since $\bigwedge^{3g-3} \pi_*\left(\widetilde{\omega}^{2}\right)$ is a line bundle on $B$, we conclude that
$$c_1\left(\bigwedge^{3g-3} \pi_*\left(\widetilde{\omega}^{2}\right)\right) \geq 0.$$
It follows that $c_1\left( \pi_*\left(\widetilde{\omega}^{2}\right)\right)\geq 0$.

Since $c_1\pi_*\left(\widetilde{\omega}^{2}\right)=\lambda_2-2\lambda(3g-3)/g$, we conclude
\[
13\lambda-\delta=\lambda_2\geq 2(3g-3)\lambda/g,
\]
which gives the desired inequality
\begin{align*}
\frac{\delta}{\lambda}\leq 7+\frac{6}{g}.
\end{align*}
\end{proof}

\subsection{Bielliptic curves}\label{S:bielliptic}
It is well-known that there exist families of bielliptic curves of genus $g$ whose slope is $8$. We sketch a construction below.
\begin{example}[Bielliptic family of slope $8$] \label{E:bielliptic}
Let $E$ be a curve of genus one. 
Consider a constant family $X:=E\times B$ and a divisor $D\subset X$ of relative degree $(2g-2)$ over
$B$. Since $K_{X}=\pi^* K_B$, adjunction gives $K_{D}-\pi^*K_B=(K_X+D)\cdot D-\pi^* K_B\cdot D=D^2$.
Thus the number of branch points of $D\ra B$ is $D^2$ by Riemann-Hurwitz formula.
Consider now the double cover $Y \ra X$ branched over $D$.
The singular fibers of $Y \ra B$ correspond to branch points of $D \ra B$.
Assuming the branch points are simple, we conclude that $\delta_{Y/B}=D^2$. On the other hand, by Mumford's formula
$$12\lambda_{Y/B}-\delta_{Y/B}=\kappa_{Y/B}=2(\omega_{X/B}+D/2)^2=D^2/2.$$ It follows that
$$\lambda_{Y/B}=D^2/8=\delta/8.$$
\end{example}
We now contrast the computation of Example \ref{E:bielliptic} with Theorem \ref{T:CH-inequality}.
Since $8>7+\frac{6}{g}$ for $g\geq 7$, Theorem \ref{T:CH-inequality} implies that the canonically embedded
general bielliptic curve
of genus $g\geq 7$ must have a non-semistable $2^{nd}$ Hilbert point. In fact, we have a more precise result.
\begin{prop}\label{P:bielliptic}
The $2^{nd}$ Hilbert point of a canonically embedded smooth bielliptic curve of genus $g\geq 7$ is non-semistable.
The $2^{nd}$ Hilbert point of a canonically embedded smooth bielliptic curve of genus $g=6$ is strictly semistable. 
\end{prop}
\begin{proof}
Consider a genus one curve $E\subset \PP^{g-2}$ embedded by a degree $g-1$ complete linear system.
There are $\binom{g}{2}-(2g-2)=\binom{g+1}{2}-3(g-1)-1$
quadrics containing $E$. It follows that a projective cone $\Cone(E)$ over $E$ in $\PP^{g-1}$ is cut out by one less
quadric than a smooth canonical curve. In fact, any smooth
quadric section of $\Cone(E)$ is a canonically embedded bielliptic
curve of genus $g$, as can be easily verified using adjunction, and conversely every canonically embedded bielliptic curve lies on such a cone. If $C\in \vert \O_{\Cone(E)}(2)\vert$, then there are $\binom{g+1}{2}-3(g-1)-1$ quadrics in 
$\HH^0(C,\I_C(2))$ that are singular at the vertex of $\Cone(E)$.
Suppose the vertex has coordinates $[0:0:\ldots:0:1]$.
If now $\rho$ is the one-parameter subgroup of $\SL(g)$ acting with weights $(-1,-1,\dots, -1, g-1)$, then
the $\rho$-weight of any monomial basis of $\HH^0(C,\I_C(2))$ is at most
\[
-2\left(\binom{g+1}{2}-3(g-1)-1\right)+2(g-1)=-g^2+7g-6=-(g-1)(g-6).
\]
Thus a bielliptic curve has non-stable $2^{nd}$ 
Hilbert point for all $g \geq 6$, and non-semistable $2^{nd}$ Hilbert point for all $g\geq 7$. 

It remains to establish the semistability of a canonically embedded smooth bielliptic curve of genus $6$. 
By above, every such curve degenerates isotrivially to a double hyperplane section 
of a cone over a genus one curve of degree $5$ in $\PP^4$. 
The semistability of this non-reduced curve follows from Kempf's results \cite{Kempf}.

\end{proof}

\subsection*{Acknowledgements} 
The first author would like to thank Aise Johan de Jong for several discussions that gave an impetus to this paper
and Anand Deopurkar for many fruitful discussions about ribbons, 
canonical curves, and slopes of sweeping families of trigonal loci. 
The second author would like to thank Sean Keel 
for several helpful conversations during the initial stages of this project.

\bibliography{HK-bib}
\bibliographystyle{alpha}

\end{document}